\documentclass[letterpaper, 11pt,  reqno]{amsart}

\usepackage{graphicx}
\usepackage{subfigure}

\usepackage{amssymb,amsmath,amsthm,amscd}

\usepackage{amsxtra, esint, bbm, enumerate}
\usepackage{mathtools} 
\usepackage{skak, bbm} 
\usepackage{ulem}

\usepackage{mparhack}

\allowdisplaybreaks[2]

\newtheorem{theorem}{Theorem} [section]

\newtheorem{lemma}[theorem]{Lemma}
\newtheorem{proposition}[theorem]{Proposition}
\newtheorem{remark}[theorem]{Remark}

\newtheorem{definition}[theorem]{Definition}

\newtheorem*{ackno}{Acknowledgment}

\DeclareMathOperator*{\supp}{supp}

\newcommand{\I}{\mathcal{I}}

\newcommand{\noi}{\noindent}
\newcommand{\Z}{\mathbb{Z}}
\newcommand{\R}{\mathbb{R}}

\newcommand{\T}{\mathbb{T}}

\newcommand{\TT}{\mathcal{T}}

\newcommand{\M}{\mathcal{M}}

\newcommand{\NB}{\mathbb{N}}

\newcommand{\FL}{\mathcal{F}L} 
\newcommand{\FLv}{\overrightarrow{\mathcal{F}L}}

\renewcommand{\H}{\mathcal{H}}

\newcommand{\ve}{\vec}

\newcommand{\F}{\mathcal{F}}
\newcommand{\IS}{\mathcal{I}}

\newcommand{\dl}{\delta}

\newcommand{\nb}{\nabla}

\newcommand{\vphi}{\vec{\phi}}

\newcommand{\eps}{\varepsilon}

\newcommand{\G}{\Gamma}
\newcommand{\ld}{\lambda}

\newcommand{\Si}{\Sigma}
\newcommand{\ft}{\widehat}
\newcommand{\Ft}{{\mathcal{F}}}

\newcommand{\dt}{\partial_t}

\renewcommand{\l}{\ell}

\renewcommand{\O}{\Omega}

\newcommand{\les}{\lesssim}
\newcommand{\ges}{\gtrsim}

\newcommand{\jb}[1]
{\langle #1 \rangle}

\newcommand{\fw}{\ft w}

\numberwithin{equation}{section}
\numberwithin{theorem}{section}

\begin{document}

\title[Norm inflation for NLW]
{A remark on norm inflation for nonlinear wave equations 
}

\author[J.~Forlano and M.~Okamoto]
{Justin Forlano and Mamoru Okamoto}

\address{
Justin Forlano\\
Maxwell Institute for Mathematical Sciences\\
Department of Mathematics\\
Heriot-Watt University\\
Edinburgh\\ 
 EH14 4AS\\
  United Kingdom \\}

\email{j.forlano@hw.ac.uk}

\address{
Mamoru Okamoto\\
Department of Mathematics \\
Graduate School of Science \\
Osaka University \\
Toyonaka\\
 Osaka 560-0043 \\
 Japan\\}

\email{okamoto@math.sci.osaka-u.ac.jp}
\date{}



\begin{abstract}
In this note, we study the ill-posedness of nonlinear wave equations (NLW). Namely, we show that 
NLW experiences norm inflation at every initial data in negative Sobolev spaces. 
This result covers a gap left open in a paper of Christ, Colliander, and Tao (2003) and extends the result by Oh, Tzvetkov, and the second author (2019) to non-cubic integer nonlinearities. 
In particular, for some low dimensional cases, we obtain norm inflation above the scaling critical regularity. We also prove ill-posedness for NLW, via norm inflation at general initial data, in negative regularity Fourier-Lebesgue and Fourier-amalgam spaces.
\end{abstract}

\maketitle

\section{Introduction}

We consider the Cauchy problem of the following nonlinear wave equation (NLW):
\begin{equation}
\begin{cases}\label{NLW1}
\dt^{2}u-\Delta u = \pm u^{k} \\
(u,\dt u)|_{t = 0} = (u_0,u_1), 
\end{cases}
\qquad ( t, x) \in \R \times \M, 
\end{equation}
\noi
where $\M=\T^d$ or $\R^d$ ($d\geq 1$) and $k \geq 2$ is an integer.

Our goal in this paper is to study ill-posedness of \eqref{NLW1} in negative Sobolev spaces.
In this regard, we recall the critical regularity associated to \eqref{NLW1} posed on $\M=\R^{d}$.
First, NLW~\eqref{NLW1} has the following scaling symmetry: given $\ld>0$, if $u$ solves \eqref{NLW1}, then $u_{\ld}(t,x)=\ld^{\frac{2}{k-1}}u(\ld t,\ld x)$ also solves \eqref{NLW1} with rescaled initial data $\ld^{\frac{2}{k-1}}(u_0 (\ld x), u_1 (\ld x))$. This scaling leaves the $\dot{H}^{s_{\text{scaling}}(d,k)}(\R^d)$-norm invariant,
where
\begin{align}
s_{\text{scaling}} (d,k) :=\frac{d}{2}-\frac{2}{k-1}.
\label{scaling}
\end{align}
Secondly, \eqref{NLW1} is invariant under the Lorentz transformation (conformal symmetry), which gives rise to the critical regularity $s_{\text{conf}}(d,k) := \frac{d+1}4-\frac1{k-1}$; see~\cite{LiSo}.
In addition, we need the condition $s \ge 0$ in order for the nonlinearity to make sense as a distribution.
Hence, the critical regularity of \eqref{NLW1} is given by
\begin{align}
\begin{split}
s_{\text{crit}}(d,k)&=\max( s_{\text{scaling}}(d,k),s_{\text{conf}}(d,k),0) \\
&= \max\bigg( \frac d2 - \frac{2}{k-1}, \frac{d+1}{4}-\frac{1}{k-1}, 0 \bigg).
\end{split} 
\label{scrit}
\end{align}
The purpose of the critical regularity for \eqref{NLW1} on $\R^d$ is that we expect (local-in-time) well-posedness in $H^s(\R^d)$ when $s>s_{\text{crit}}(d,k)$ and ill-posedness, due to some instability, when $s<s_{\text{crit}}(d,k)$. This heuristic provided by \eqref{scrit} is also instrumental in the well-posedness theory of \eqref{NLW1} on periodic domains $\M=\T^{d}$, despite the lack of scaling and conformal symmetries in this setting. 

We now survey the well-posedness theory for \eqref{NLW1}, specifically restricting our attention to local-in-time results. 
Well-posedness of \eqref{NLW1} above the critical regularity $s_{\text{crit}}(d,k)$ was studied in \cite{Kap, LiSo, KeelTao, Tao}.
Moreover, ill-posedness of \eqref{NLW1} below the critical regularity has been studied in \cite{Lin1, Lin2, CCT, Leb, BTz1, Xia, Tz1, OOTz}.
In particular, 
Christ, Colliander, and Tao \cite{CCT} proved norm inflation for \eqref{NLW1} on $\R^d$ when\footnote{They considered the nonlinearity $\pm |u|^{k-1}u$ instead of $\pm u^k$. Moreover, they proved that the solution map
fails to be uniformly continuous when $s<0$ and for the defocusing case. See
also \cite{LiSo} for the ill-posedness result in the focusing case.}: (i) $k\geq 2$ and 
\begin{align*}
s_{\text{scaling}}(1,k)<s< \frac 12 - \frac 1k,
\end{align*}
and (ii) for either odd integer $k\geq 3$ or $k\geq k_0+1$ for integer $k_0 >\frac{d}{2}$ and 
\begin{align*}
s\leq -\frac{d}{2}\quad \text{or}\quad  0<s<s_{\text{scaling}}(d,k).
\end{align*}
Applying the argument in \cite[Corollary 7]{CCT}, which uses the finite speed of propagation for \eqref{NLW1} to deduce norm inflation in dimension $d\geq 2$ from norm inflation in $d=1$, the result of (ii) extends to norm inflation for any $k\geq 2$, $s\leq -\frac{1}{2}$, and $s<s_{\text{scaling}}(1,k)$.
Here, norm inflation (at the trivial initial condition $(u_0,u_1)=(0,0)$) means that given any $\eps>0$, there exists a solution $u_\eps$ to \eqref{NLW1} and $t_\eps \in (0,\eps)$ such that
\begin{align}
 \| (u_\eps(0), \dt u_\eps(0)) \|_{\H^s(\M)} < \eps \qquad \text{ and } 
\qquad \| u_\eps(t_\eps)\|_{H^s(\M)} > \eps^{-1}, 
\label{NI1}
\end{align}
where $\H^s (\M) := H^s(\M) \times H^{s-1}(\M)$.
This phenomenon is a stronger notion of ill-posedness than the discontinuity of the solution map at zero.
In particular, the result in \cite{CCT} leaves open the question of norm inflation for NLW~\eqref{NLW1} when 
\begin{align}
-\frac 12 <s< \min (s_{\text{scaling}}(1,k),0). \label{gap}
\end{align}
In the context of \eqref{NLW1} on $\T^{3}$ for $0<s<s_{\text{scaling}}(3,k)$, Xia~\cite{Xia} generalized \eqref{NI1} to norm inflation based at general initial data (see \eqref{NI2} below).
In \cite{OOTz}, Oh, Tzvetkov, and the second author proved norm inflation at general initial data for the cubic NLW ($k=3$) when $d\geq 2$ and $s<0$.\footnote{We point out that this norm inflation result was proved as a basic ingredient for the
 main purpose of the paper~\cite{OOTz}; namely, to study the approximation property of solutions to the renormalized cubic NLW on $\T^2$ with rough, random initial data distributed according to the Gaussian free field.} For the particular case $k=3$ and $d\geq 2$, this result extends the norm inflation at zero in \cite{CCT} to norm inflation at general initial data.

Our aim in this paper is to prove norm inflation at general initial data for \eqref{NLW1} in negative Sobolev spaces, thus filling the remaining gap left open in \eqref{gap}.
The following is our main result.

\begin{theorem}\label{THM:NI}
Given $d \in \NB$, 
let  $\M = \R^d$ or $\T^d$.
Suppose that $k \ge 2$ is an integer and $s<0$.
Fix $(u_0, u_1) \in \H^s(\M)$.
Then, 
given any $\eps > 0$, 
there exist a solution $u_\eps$ to \eqref{NLW1} on $\M$
and $t_\eps  \in (0, \eps) $ such that 
\begin{align}
 \| (u_\eps(0), \dt u_\eps(0))  - (u_0, u_1) \|_{\H^s(\M)} < \eps \qquad \text{and} 
\qquad \| u_\eps(t_\eps)\|_{H^s(\M)} > \eps^{-1}.
\label{NI2}
\end{align}
\end{theorem}

Theorem~\ref{THM:NI} thus closes the remaining gap in \eqref{gap} and, in the case $s<0$ and $k\neq 3$ (in view of \cite{OOTz}), extends the result in \cite{CCT} to norm inflation based at any initial condition.
When $(u_0,u_1)=(0,0)$, Theorem \ref{THM:NI} is reduced to the usual norm inflation at zero initial data stated in \eqref{NI1}.
As a corollary to Theorem \ref{THM:NI}, we obtain that the solution map to \eqref{NLW1}: $(u_0,u_1) \in \H^s(\M) \mapsto (u,\dt u) \in C([-T,T];\H^s(\M))$ is discontinuous everywhere in $\H^s(\M)$, for $s<0$.

We now describe two approaches to proving norm inflation for Cauchy problems. The first is the approach used in \cite{CCT} which is based on studying low-to-high energy transfer in the associated dispersionless (ODE) model and scaling analysis. By avoiding the scaling analysis, Burq and Tzvetkov \cite{BTz1} proved norm inflation as in \eqref{NI1} for the cubic NLW on three-dimensional compact Riemannian manifolds when $0<s<s_{\text{scaling}}(3,3)$.
In particular, the argument in \cite{Xia} is also based on this method.
The second method is a Fourier analytic approach introduced by Bejenaru and Tao~\cite{BT} and developed further by Iwabuchi and Ogawa \cite{IO}; see also~\cite{Kishimoto, O17}.
Our proof of Theorem~\ref{THM:NI} uses this method and follows the presentation by Oh~\cite{O17}, which we now briefly describe.
We begin with a reduction: we may assume the initial data $(u_0,u_1)$ are sufficiently regular by a density argument. 
 The key idea is to express a solution $u_{\eps}$ to \eqref{NLW1} in terms of a power series expansion in the initial data and to show that one of the terms in the expansion dominates all the others. 
  More specifically, we write $u_{\eps}$ as the following power series expansion:
\begin{align} 
u_{\eps}= \sum_{j=0}^{\infty} \Xi_{j}(u_{\eps}(0),\dt u_{\eps}(0)), \label{psexp}
\end{align}
where $\{\Xi_{j}\}_{j=0}^{\infty}$ are multilinear operators in the linear solution $S(t)(u_{\eps}(0),\dt u_{\eps}(0))$ of (increasing) degree $kj+1$. They are precisely the successive new terms added to a Picard iteration expansion of $u_{\eps}$. 
We define the initial data for $u_{\eps}$ by
\begin{align}
(u_{\eps}(0),\dt u_{\eps}(0))=(u_0, u_1)+(\phi_{0,\eps},\phi_{1,\eps}), \label{shiftdata}
\end{align}
 where the perturbations $(\phi_{0,\eps},\phi_{1,\eps})$ are chosen so that:
\begin{enumerate}[(i)]
\item $(\phi_{0,\eps},\phi_{1,\eps})$ converges to $(0,0)$ in $\H^{s}(\M)$, as $\eps\rightarrow  0$,
\item there exists times $t_{\eps}\rightarrow 0$ as $\eps\rightarrow 0$ such that the second Picard iterate $\Xi_{1}(\phi_{0,\eps},\phi_{1,\eps})$ dominates in \eqref{psexp}; namely, 
\begin{align*}
\| u_{\eps}(t_{\eps})\|_{H^{s}(\M)} \ges \big\| \Xi_{1}(\phi_{0,\eps},\phi_{1,\eps})(t_{\eps}) \big\|_{H^{s}(\M)} \rightarrow \infty,
\end{align*} 
as $\eps \rightarrow 0$.
\end{enumerate} 
 These ingredients then lead to norm inflation based at $(u_0,u_1)$ as in \eqref{NI2}. The mechanism responsible for the instability in (ii) is the high-to-low transfer of energy, which is specifically exploited by the choice of $(\phi_{0,\eps},\phi_{1,\eps})$. 
 We note that although we work in rough topologies, the functions $u_{\eps}$ are smooth and hence there is no issue in making sense of the power series expansion \eqref{psexp}. 
In \cite{O17}, the operators $\Xi_{j}$ are indexed using trees, which allows to directly treat the nonlinear estimates without an induction. 

As it is based on exploiting high-to-low energy transfer in the nonlinearity, the Fourier analytic approach works well in negative Sobolev spaces. 
Indeed, for the case of nonlinear Schr\"{o}dinger equations (NLS), this method was used in \cite{IO, Kishimoto, O17} to fill a similar gap left in \cite{CCT} of norm inflation for NLS in negative Sobolev spaces. 
 However, it does rely on the translation invariance of the underlying space $\mathcal{M}$, making it  unsuitable for the case of more general domains. 
See also \cite{M, CDS, AC, CK} for ill-posedness results of NLS. The idea of exploiting a high-to-low energy transfer was used in \cite{BTz} to prove failure of $C^{2}$-smoothness of the solution map for the BBM equation.


For $k\in \{2,3,4\}$ in $d=1$ and $k=2$ in $d=2,3$, Theorem~\ref{THM:NI} yields norm inflation at general initial data above the scaling critical regularity $s_{\text{scaling}} (d,k)$ defined in \eqref{scaling}. 
This phenomenon of norm inflation above the scaling critical regularity has also been observed for the cubic fractional NLS~\cite{CP} and quadratic NLS~\cite{IO, Kishimoto, Ok}. In this regime, it is essential to exploit resonant interactions in the nonlinearity.
In the aforementioned papers, the choice of the initial data $(\phi_{0,\eps},\phi_{1,\eps})$ in \eqref{shiftdata} 
(with $(u_0,u_1)=(0,0)$) 
only activates (nearly) resonant contributions in the second Picard iterate $\Xi_{1}(\phi_{0,\eps},\phi_{1,\eps})$. However, for the case of NLW~\eqref{NLW1}, our analysis of the second Picard iterate is more subtle since our choice of perturbation $(\phi_{0,\eps},\phi_{1,\eps})$ requires us to also handle nonresonant contributions. To show that the resonant part is dominant, we need to take the existence time a bit longer. See Proposition \ref{PROP:INSTAB}. 

We also note that the argument in \cite[Corollary 7]{CCT} is not applicable for deducing norm inflation at general initial data in dimensions $d\geq 2$ from norm inflation at general initial data in dimension $d=1$. Thus, we cannot simply deduce Theorem~\ref{THM:NI} from the corresponding result in one dimension.

\begin{remark}\rm
In this paper, we focus on the estimate for the norm of $u$.
However, by the same argument as in Section \ref{SEC:3} below, we can show the
growth of the norm of $\dt u$ as well as of $u$ in \eqref{NI2}:
\begin{align*}
\| \dt u_\eps (t_\eps) \|_{H^{s-1}} > \eps^{-1}.
\end{align*}
See Remark~\ref{RMK:dt} for more details.
\end{remark}

\begin{remark}\rm
By a straightforward modification, the same norm inflation result as in Theorem \ref{THM:NI} holds for the following nonlinear Klein-Gordon equation:
\begin{equation}
\begin{cases}\label{NLKG}
\dt^{2}u-\Delta u + u = \pm u^{k} \\
(u,\dt u)|_{t = 0} = (u_0,u_1), 
\end{cases}
\qquad ( t, x) \in \R \times \M.
\end{equation}
See Remark \ref{REM:NLKGlb} for a further discussion.
\end{remark}

\begin{remark}\rm
Theorem~\ref{THM:NI} completes the ill-posedness theory for NLW~\eqref{NLW1} in negative regularities $s<0$. The situation however is not complete in positive regularities. In particular, for $d\geq 2$ and $0<s<s_{\text{conf}}(d,k)$, the focusing NLW (corresponding to the $+\,$sign in \eqref{NLW1}) on $\R^d$ has explicit solutions with arbitrarily small $\H^{s}(\R^d)$-norm which blow-up in arbitrarily small time; see~\cite{LiSo} and \cite[Exercise 3.67]{TAO}. In contrast, it is not known if there are similarly behaving blow-up solutions in the defocusing case (the $-\,$sign in \eqref{NLW1}) when $s_{\text{scaling}}(d,k)<s<s_{\text{conf}}(d,k)$.
\end{remark}

\begin{remark}\rm
There are also other approaches to proving ill-posedness results, for instance, the work of Lebeau~\cite{Leb} and Carles and collaborators~\cite{AC, CDS, CK, BC}. The results they obtain are stronger than norm inflation at zero and demonstrate a loss of regularity:  there are smooth solutions for which the norm in the second expression in \eqref{NI2} can be replaced by $H^{\sigma}$ for any (or some) $\sigma \in \R$.
Kishimoto~\cite{Kishimoto} also proves infinite loss of regularity results using the Fourier analytic approach. It would be of interest to investigate this loss of regularity phenomenon for the NLW~\eqref{NLW1} in negative Sobolev spaces.
\end{remark}

While our main interest in this paper is to close the gap in \eqref{gap} and extend previous results to norm inflation at general initial data in $\H^{s}(\M)$, our proof admits an easy extension to more general spaces of functions and hence, to a more general result than that in Theorem~\ref{THM:NI}. 

We now introduce these spaces, which are the Fourier-Lebesgue and Fourier-amalgam spaces.
We use $\mathcal{S}(\M)$ to denote the space of Schwartz functions if $\M=\R^d$ or the space of $C^{\infty}$-functions if $\M=\T^{d}$.
Given $s\in \R$ and $1\leq q\leq \infty$, we define the Fourier-Lebesgue space $\F L^{s,q}(\M)$ as the closure of $\mathcal{S}(\M)$ under the following norm:
\begin{align*}
\|f \|_{\F L^{s,q}(\M)} = \big\| \jb{\xi}^s \ft f \big\|_{L^{q}(\ft \M)},
\end{align*}
where $\jb{\, \cdot \,}:=(1+|\cdot|^2)^{\frac{1}{2}}$ and $\ft \M$ denotes the Pontryagin dual of $\M$, 
i.e.,~
\begin{align}
\ft \M = \begin{cases}
 \R^d & \text{if }\M = \R^d, \\
 \Z^d & \text{if } \M = \T^d.
\end{cases}
\label{dual}
\end{align}

\noi
When $\ft \M = \Z^d$, 
we endow it with the counting measure. We note that $\FL^{s,2}(\M)=H^{s}(\M)$, and $\FL^{s,q_1}(\T^d)\subseteq \FL^{s,q_2}(\T^d)$ when $1\leq q_1 \leq q_2\leq \infty$. This last property implies that the spaces $\FL^{s,q}(\T^{d})$ with $q>2$ are wider than $H^{s}(\T^d)$, for any $s\in \R$.
On the Euclidean spaces there is, in general, no relation between the spaces $\FL^{s,q}(\R^d)$ for fixed $s\in \R$. 
The Fourier-Lebesgue spaces are encapsulated within a wider class of function spaces known as the Fourier-amalgam spaces. Given $s\in \R$ and $1\leq p,q\leq \infty$, 
we define the Fourier-amalgam space $\fw^{p,q}_{s}(\R^d)$ as the closure of $\mathcal{S}(\R^d)$ under the norm
\begin{align}
\|f\|_{\fw^{p,q}_{s}(\R^d)}:=\big\| \jb{n}^{s} \| \chi_{n+Q}(\xi)\ft f(\xi)\|_{L_{\xi}^{p}(\R^d)}\big\|_{\l^q_{n}(\Z^d)},
\label{fw}
\end{align}
where $Q:=[-\frac{1}{2},\frac{1}{2})^{d}$ and $\chi_{n+Q}$ is the indicator function on the set $n+Q$, for $n\in \Z^d$. In the special case $p=q$, we have $\fw^{q,q}_{s}(\R^d)=\FL^{s,q}(\R^d)$. Moreover, when $p=2$, $\fw^{2,q}_{s}(\R^d)=M^{2,q}_{s}(\R^d)$, where $M^{2,q}_{s}(\R^d)$ is a modulation space. We point out that the Fourier-amalgam spaces $\fw^{p,q}_{s}(\R^d)$ are the Fourier image of the classical weighted Wiener-amalgam spaces $W(L^p,L_{s}^q)$; see for example \cite{Fei}.
Our interest in the Fourier-amalgam spaces rests solely in the Euclidean space setting where $\M=\R^d$, since in the periodic case, we have
\begin{align*}
\fw^{p,q}_{s}(\T^d)=M^{2,q}_{s}(\T^d)=\FL^{s,q}(\T^d), 
\end{align*}
for any $1\leq p,q \leq \infty$ and $s\in \R$. 

\begin{remark}\rm
In \cite{IO}, Iwabuchi and Ogawa use the space $(M^{2,1})_A(\R^d)$ defined through the following norm
\begin{align*}
\|f\|_{(M^{2,1})_A (\R^d)}:=\sum_{ n \in A\Z^d}\|\chi_{n+Q_{A}}(\xi)  \ft f(\xi) \|_{L^{2}_{\xi}(\R^d)},
\end{align*}
where $Q_{A}:=[-\frac{A}{2},\frac{A}{2})^{d}$, and $A>0$ is a dyadic number, to verify the convergence of the power series expansions. These spaces satisfy $(M^{2,1})_{A}(\R^d) \simeq_{A} (M^{2,1})_1 (\R^d)$. We note that $(M^{2,1})_1(\R^d)=\fw_{0}^{2,1}(\R^d)$, and the algebra property of $(M^{2,1})_{A}(\R^d)$ can be viewed as a consequence of the algebra property of the spaces $\fw^{p,1}_{0}(\R^d)$, for $1\leq p \leq \infty$.
\end{remark}

As in the $L^{2}$-based case above, we can associate to the NLW~\eqref{NLW1} the scaling critical regularity 
\begin{align}
s_{\text{scaling}}(d,k,q):=d\bigg( 1-\frac{1}{q}\bigg)-\frac{2}{k-1}, \label{scalingFLq}
\end{align}
in the Fourier-Lebesgue spaces $\FL^{s,q}(\M)$. Unfortunately, the Fourier-amalgam spaces do not interact well with respect to scaling. However, H\"{o}lder's inequality implies $\fw^{p,q}_{s}(\R^d)$ is a weaker norm than $\FL^{s,q}(\R^d)$ when $p\leq q$ and for any $s\in \R$, so we expect \eqref{scalingFLq} to essentially be a scaling critical regularity for \eqref{NLW1} in Fourier-amalgam spaces $\fw^{p,q}_{s}(\R^d)$.

There are very few well-posedness results for NLW~\eqref{NLW1} in these spaces. 
In Fourier-Lebesgue spaces, 
well-posedness for the quadratic NLW ($k=2$ in \eqref{NLW1}) in $\FL^{s,q}(\R^{2})$ for $2\leq q<\infty$ and $s>1+\frac{3}{2q'}$ follows from \cite[Equation (9)]{GrigT}; in fact, the main results in \cite{Grunrock, GrigN, GrigT} are for the quadratic derivative NLW in $\FL^{s,q}(\R^d)$ for $d=2,3$.
The well-posedness study of nonlinear dispersive equations in Fourier-amalgam spaces was recently initiated by the first author and T.~Oh~\cite{FO}, in the context of the one-dimensional cubic NLS on $\R$, and to the best of the authors' knowledge, there are no well-posedness results for NLW~\eqref{NLW1} in Fourier-amalgam spaces. As for ill-posedness results,
norm inflation in Fourier-Lebesgue spaces for NLS-type equations have been obtained in \cite{CK, Kishimoto, BC} and also in modulation spaces in \cite{BC}. For NLW~\eqref{NLW1}, our second main result is norm inflation at general initial data in negative regularity Fourier-Lebesgue and Fourier-amalgam spaces.

\begin{theorem}\label{THM:FLq}
Let $d \in \NB$ and suppose that $k \ge 2$ is an integer. Given $1\leq p, q\leq \infty$ and $s<0$, we define
\begin{align}
X^{p,q}_{s}(\M):=
\begin{cases} 
      \fw^{p,q}_{s}(\R^d) & \text{if} \,\,\, \,\M=\R^d,  \\
      \FL^{s,q}(\T^d) & \text{if}\,\,\,\, \M=\T^d \,\,\,  \text{and}\,\,\, \,p=q.
   \end{cases} \label{Xspace}
\end{align}
Fix $(u_0, u_1) \in X^{p,q}_{s}(\M)$.
Then, 
given any $\eps > 0$, 
there exist a solution $u_\eps$ to \eqref{NLW1} on $\M$
and $t_\eps  \in (0, \eps) $ such that 
\begin{align}
 \| (u_\eps(0), \dt u_\eps(0))  - (u_0, u_1) \|_{X^{p,q}_{s}(\M)} < \eps \qquad \text{and} 
\qquad \| u_\eps(t_\eps)\|_{X^{p,q}_{s}(\M)} > \eps^{-1}.
\notag
\end{align}
\end{theorem}

In the body of this paper, we detail the proof of Theorem~\ref{THM:NI}. The proof of Theorem~\ref{THM:FLq} follows from minor modifications of the aforementioned proof and these details are contained in Section~\ref{app:amalgam}.
Due to the local-in-time nature of the analysis in this paper, the sign of the nonlinearity in \eqref{NLW1} does not play any role.
Hence, we only consider the $+\,$sign in the following.
Moreover, in view of the time reversibility of the equation, we focus only on positive times. 


\section{Power series expansion indexed by trees}
\label{SEC:2}

In this section, we show the well-posedness in the Fourier-Lebesgue space and exploit power series expansions.
We define
\begin{align}
\FLv^{s, q}(\M) := \F L^{s,q}(\M) \times \F L^{s-1,q}(\M),
\notag
\end{align}
and, for convenience, write $\FL^q (\M) := \FL^{0,q}(\M)$ and $\FLv^q (\M) := \FLv^{0,q}(\M)$.

Let $S(t)$ denote the linear wave propagator:
\begin{align}
S(t)(\ve{u}_0)=  S(t)(u_0,u_1)=\cos(t|\nb|) u_0+\frac{\sin (t|\nb|)}{|\nb|} u_1
\label{linsol}
\end{align}
and let $\IS$ denote the $k$-linear Duhamel operator 
\begin{align}
\IS[u_1,\ldots, u_k] (t)
:= \int_{0}^{t} \frac{\sin ((t-t')|\nb|)}{|\nb |}\bigg[ \prod_{j=1}^{k} u_j(t') \bigg]dt'.
\label{forcing}
\end{align}
Writing $\IS^k [u]=\IS[u,\ldots, u]$, we have the following Duhamel formulation of \eqref{NLW1}: 
\begin{align}
u(t)= S(t)(\ve{u}_{0})+\IS^k[u](t). \label{duhamel}
\end{align}
We use the convention
\begin{align*}
\frac{\sin (t |0|)}{|0|}=t.
\end{align*}
For $0\leq t \leq 1$, we have
\begin{align}
\| S(t)\vec{u}_0\|_{H^s} \leq \|u_0\|_{H^s}+t\|u_1\|_{H^{s-1}}\leq \|\vec{u}_0\|_{\H^{s}}.
\label{linestHs}
\end{align}

Second, we recall the following definitions and terminology used in \cite{O17} to describe $k$-ary trees.
 
 \begin{definition} \label{DEF:tree} \rm
(i) Given a partially ordered set $\TT$ with partial order $\leq$, 
we say that $b \in \TT$ 
with $b \leq a$ and $b \ne a$
is a child of $a \in \TT$,
if  $b\leq c \leq a$ implies
either $c = a$ or $c = b$.
If the latter condition holds, we also say that $a$ is the parent of $b$.

\smallskip 

\noi
(ii) 
A tree $\TT$ is a finite partially ordered set,  satisfying
the following properties\footnote{We do not identify two trees even if there is an
order-preserving bijection between them.}:
\begin{itemize}
\item Let $a_1, a_2, a_3, a_4 \in \TT$.
If $a_4 \leq a_2 \leq a_1$ and  
$a_4 \leq a_3 \leq a_1$, then we have $a_2\leq a_3$ or $a_3 \leq a_2$,

\item
A node $a\in \TT$ is called terminal, if it has no child.
A non-terminal node $a\in \TT$ is a node 
with exactly $k$ children,

\item There exists a maximal element $r \in \TT$ (called the root node) such that $a \leq r$ for all $a \in \TT$,

\item $\TT$ consists of the disjoint union of $\TT^0$ and $\TT^\infty$,
where $\TT^0$ and $\TT^\infty$
denote the collections of non-terminal nodes and terminal nodes, respectively.

\end{itemize}
\noi
(iii) Let $\boldsymbol{T}(j)$ denote the set of all trees with $j$ non-terminal nodes
\end{definition}

Note that a given $k$-ary tree $\TT \in \pmb{T}(j)$ has $kj+1$ nodes. This follows from the fact that the number of non-terminal and terminal nodes of $\TT$ are $j$ and $(k-1)j+1$, respectively, where $j\in \mathbb{N} \cup \{0\}$.

We have the following basic combinatorial property for $k$-ary trees. The proof is a straightforward adaptation of the one in \cite[Lemma 2.3]{O17} for ternary trees ($k=3$).

\begin{lemma}\label{lemma:counting} 
There exists a constant $C_{0}>0$ such that 
\begin{equation} \label{tjcnt}
|\pmb{T}(j)| \leq C_{0}^{j}.
\end{equation}
\end{lemma}

 For fixed $\vphi\in \FLv^{1} (\M)$, we associate to a given tree $\mathcal{T}\in \mathbf{T}(j)$, a space-time distribution $\Psi_{\vphi}(\TT)\in \mathcal{D}'([0,T]\times \M)$ as follows: we replace a non-terminal node by the Duhamel integral operator $\I$ with its $k$ arguments as children and we replace all terminal nodes by the linear solution $S(t)\vphi$. We then define
 \begin{align}
\Xi_{j}(\vphi)=\sum_{\mathcal{T}\in \mathbf{T}(j)} \Psi_{\vphi}(\mathcal{T}).
\label{sum}
\end{align}
For example,
\begin{align*}
\Xi_{0}(\vphi)=S(t)\vphi \quad \text{and} \quad \Xi_{1}(\vphi)=\I[ S(t)\vphi, \ldots, S(t)\vphi].
\end{align*}

\noi
The multilinear operators $\Xi_{j}$ satisfy the following estimates.
We use short-hand notations such as $C_T \FL^p = C([0, T]; \FL^p(\M))$ for $T>0$. The following estimates will be used to show convergence of the power series expansion and is similar to \cite[Lemma 2.5]{O17}.

\begin{lemma}\label{LEM:Xiestimates} There exists $C>0$ such that the following hold:
Given $\ve \phi \in \FLv^{1}(\M)$, $j\in \mathbb{N}$, $\vec{\psi}\in \FLv^{1}(\M) \cap \H^{0}(\M)$, and $0<T\leq 1$, we have 
\begin{align}
\| \Xi_{j}(\ve{\phi})(T)\|_{\FL^{1}} &\leq C^{j}T^{2j}\|\ve{\phi}\|_{\FLv^{1}}^{(k-1)j+1}, \label{FL1}\\
\| \Xi_{j}(\ve{\psi})(T)\|_{\FL^{\infty}} &\leq C^{j}T^{2j}\|\ve{\psi}\|_{\FLv^{1}}^{(k-1)j-1}  \|\ve{\psi}\|_{\H^{0}}^{2}.\label{FLinfty}\
\end{align}
\end{lemma}
\begin{proof}
For $\vec{\phi}=(\phi_0,\phi_1)$, we have from \eqref{linsol} and $0<t\leq T\leq 1$,
\begin{align}
\| S(t)(\vec{\phi})\|_{C_T\FL^{1}}\leq \| \phi_0 \|_{\FL^{1}}+T\|\phi_{1}\|_{\FL^{-1,1}}\leq \|\vec{\phi}\|_{\FLv^{1}}.
\label{linest}
\end{align}
As $|\sin t|\leq |t|$ for every $t\in \R$, \eqref{forcing} and the algebra property of $\FL^{1}(\M)$ imply
\begin{align}
\| \I[u_1,\ldots, u_k]\|_{C_T \FL^1} \leq \bigg(\int_{0}^{T} (T-t) dt\bigg) \prod_{j=1}^{k}\|u_j\|_{C_T \FL^1}\leq CT^2 \prod_{j=1}^{k}\|u_j\|_{C_T \FL^1}. \label{Iesti}
\end{align}
For a fixed $\mathcal{T} \in \mathbf{T}(j)$, $\Psi_{\vec{\phi}}\,(\mathcal{T})$ is essentially $j$ iterated compositions of the operator $\I^{k}[S(t)\vec{\phi}]$, with $(k-1)j+1$ terms $S(t)\vec{\phi}$. Hence, \eqref{FL1} follows from \eqref{sum}, \eqref{tjcnt}, \eqref{Iesti}, and \eqref{linest}.
Likewise, \eqref{FLinfty} follows similarly in addition to using Young's inequality. 
\end{proof}

We now justify the power series expansion for solutions to \eqref{duhamel}.

\begin{lemma}\label{LEM:local}
Let $k\geq 2$ be an integer and $M>0$. Then, for any $0<T\ll \min (M^{-\frac{k-1}{2}},1)$ and $\vec{u}_0 \in \FLv^{1}(\M)$ with $\| \vec{u}_0 \|_{\FLv^{1}}\leq M$, the following holds:
\begin{itemize}
\item[(i)] There exists a unique solution $u\in C([0,T];\F L^{1}(\M))$ satisfying $(u,\dt u)|_{t=0}=\vec{u}_0$  to \eqref{duhamel}.
\item[(ii)] The solution $u$ in \textup{(i)} may be expressed as 
\begin{align}
u= \sum_{j=0}^{\infty}\Xi_{j}(\vec{u}_{0})=\sum_{j=0}^{\infty}\sum_{\mathcal{T}\in \mathbf{T}(j)} \Psi_{\vec{u}_{0}}(\mathcal{T}), \label{uexpand}
\end{align}
where the series converges absolutely in $C([0,T];\FL^{1}(\M))$.
\end{itemize}
\end{lemma}

\begin{proof}
We begin with (i). We define 
\begin{align*}
\G[u](t):= S(t)(\vec{u}_0)+\I^{k}[u](t).
\end{align*}
Then, \eqref{linest} and \eqref{Iesti} imply
\begin{align*}
\|\G[u]\|_{C_T \FL^{1}}
\leq \| \vec{u}_0 \|_{\FLv^{1}}+CT^{2}\|u\|_{C_T \FL^{1}}^{k}.
\end{align*}
Thus, for $0<T\leq 1$ such that $CT^{2}M^{k-1}\ll 1$, $\G$ maps the ball $B_{2M}:=\{ v\in C([0,T];\FL^{1}(\M)) \, : \, \|v\|_{C_T \FL^{1}}\leq 2M\}$ into itself. In view of the multilinearity of $\I$, we may reduce $T$ further to ensure that $\G$ is in fact a strict contraction on $B_{2M}$. The contraction mapping theorem and an a posteriori continuity argument completes (i). We move onto verifying (ii). We fix $0<T\leq 1$ such that $CT^{2}M^{k-1}\ll 1$ and fix $\eps>0$.
From \eqref{FL1}, we see that the sum in \eqref{uexpand} converges absolutely in $C([0,T];\FL^{1}(\M))$ and hence there exists an integer $J_1 \geq 0$ such that for every $J\geq J_1$, 
\begin{align}
\| U-U_{J}\|_{C_{T}\FL^1}<\frac{\eps}{3},
\label{U1}
\end{align}
where
\begin{align*}
U_J := \sum_{j=0}^{J}\Xi_{j}(\vec{u}_0)  \quad \text{and} \quad U:=\sum_{j=0}^{\infty}\Xi_{j}(\vec{u}_0).
\end{align*}
In particular, $U, U_{J}\in B_{2M}$ for any $J\in \mathbb{N}$.
From (i), $\G$ is continuous from $B_{2M}$ into itself and hence there exists an integer $J_2 \geq 0$ such that for every $J\geq J_2$,
\begin{align}
\|\G[U_{J}]-\G[U]\|_{C_{T}\FL^1} <\frac{\eps}{3}.
\label{U3}
\end{align}
Now for a fixed integer $J\geq 1$, we consider the difference $U_J-\G[U_J]$.
We have 
\begin{align*}
U_J-\G[U_J]& = \sum_{j=1}^{J}\Xi_{j}(\vec{u}_{0})-\sum_{0\leq j_1,\ldots,j_k\leq J}\I[\Xi_{j_1}(\vec{u}_0),\ldots, \Xi_{j_{k}}(\vec{u}_0)] \\
& = -\sum_{\l=J+1}^{kJ+1}\sum_{\substack{0\leq j_1,\ldots,j_k\leq J \\ j_1+\cdots +j_k=\l-1}}\I[\Xi_{j_1}(\vec{u}_0),\ldots, \Xi_{j_{k}}(\vec{u}_0)].
\end{align*}
Using \eqref{Iesti}, \eqref{FL1}, and crudely estimating the sums, we obtain
\begin{align*}
\| U_J-\G[U_J]\|_{C_T \FL^1}
&\leq CT^{2} \sum_{\l=J+1}^{kJ+1} \sum_{\substack{0\leq j_1,\ldots,j_k\leq J \\ j_1+\cdots +j_k=\l-1}} \prod_{m=1}^{k} \|\Xi_{j_m}(\vec{u}_0)\|_{C_T \FL^1} \\
& \leq CT^{2}M^{k}J^{k}\sum_{\l=J+1}^{\infty}(CT^{2}M^{k-1})^{\l-1} \\
& \leq CM J^{k}(CT^{2}M^{k-1})^{J}.
\end{align*}
Thus, there exists an integer $J_3 \geq 0$ such that for every $J\geq J_3$,
\begin{align}
\| U_{J}-\G[U_{J}]\|_{C_T \FL^1} <\frac{\eps}{3}. 
\label{U2}
\end{align}
With $J:=\max_{\l=1,2,3} J_{\l}$, \eqref{U1}, \eqref{U3}, and \eqref{U2} imply  
\begin{align*}
\|U-\G[U]\|_{C_{T}\FL^1}&\leq \|U-U_J\|_{C_{T}\FL^1}+\| U_J-\G[U_J]\|_{C_{T}\FL^1} 
+\|\G[U_J]-\G[U]\|_{C_{T}\FL^1} \\
& <\eps.
\end{align*}
Thus, $U=\G[U]$ and, by uniqueness, we conclude $u=U$.
\end{proof}

\section{Norm inflation for NLW}
\label{SEC:3}

In this section, we present the proof of Theorem \ref{THM:NI} by establishing the following proposition.

\begin{proposition} \label{PROP:NI}
Let $\M=\R^d$ or $\T^d$, $k\geq 2$ be an integer, $s<0$, and fix $u_0,u_1\in \mathcal{S}(\M)$. Then, for any $n\in \mathbb{N}$, there exists a smooth solution $u_n$ to the NLW~\eqref{NLW1} and $t_{n}\in (0, \frac{1}{n})$ such that
\begin{align}
\|(u_n (0), \dt u_{n}(0))-(u_0,u_1)\|_{\H^{s}(\M)}<\frac 1n
\qquad \text{and} \qquad
\|u_{n}(t_n)\|_{H^{s}(\M)}>n. 
\label{nexplosion}
\end{align}
\end{proposition}

From density and diagonal arguments,
Theorem \ref{THM:NI} follows from Proposition \ref{PROP:NI}.
See \cite{Xia, O17} for the details.

Thus, the remaining part of this paper is devoted to the proof of Proposition \ref{PROP:NI}. 
While the argument closely follows \cite[Section 3]{O17}, we will detail it here 
in order to make this paper self-contained.
In the following, we fix $\vec{u}_0 =(u_0,u_1)$ with $u_0, u_1 \in \mathcal{S}(\M)$.
In Subsection \ref{SUBSEC:31}, we prove multilinear estimates for each term in the power series expansion.
Moreover, by observing high-to-low energy transfer and resonant interaction, we show a crucial lower bound for the first multilinear term.
We then present the proof of Proposition \ref{PROP:NI} in Subsection \ref{SUBSEC:32}.

\subsection{Multilinear estimates}
\label{SUBSEC:31}

In this subsection, we state the multilinear estimates on $\Xi_j$.
Moreover, we show that the first multilinear term $\Xi_1$ is the leading part in the power series expansion in negative Sobolev spaces. 

Let $\chi_K$ denote the indicator function of a subset $K \subset \ft{\M}$, where $\ft{\M}$ is as in \eqref{dual}.
Set $e_1 := (1,0,\dots, 0) \in \ft{\M}$.
Given $n\in \NB$, let $N=N(n)\gg 1$ to be chosen later. We set $\ve{\phi}_{n}=(\phi_{0,n}, 0)$ by
\begin{align}
\ft \phi_{0,n}=R \chi_{\Omega},
\label{data}
\end{align}
where $R=R(N)\geq 1$,
\begin{align}
\Omega= \bigcup_{\eta \in \Sigma} (\eta+Q_{A}), \label{omega}
\end{align}
$Q_{A}=[-\frac A2, \frac A2)^{d}$, $A\gg 1$, and
\begin{align}
\Sigma=\{-2Ne_1, -Ne_1, Ne_1,2Ne_1\}.
\label{Sig}
\end{align}
We require $N, R,$ and $A$ to be chosen so that 
 \begin{align}
\|\ve{u}_0\|_{\FLv^{1}}\ll RA^{d} \qquad \text{and} \qquad A\ll N,
\label{sizedatau}
\end{align}
where the last condition ensures that $\Omega$ in \eqref{omega} is a disjoint union. 
Notice that \eqref{data} and \eqref{omega} imply 
\begin{align}
\|\ve{\phi}_n\|_{\FLv^{1}}
=\|\phi_{0,n}\|_{\FL^{1}} \sim RA^{d} \qquad \text{and}  \qquad \|\ve{\phi}_{n}\|_{\H^{s}}\sim RA^{\frac d2}N^{s},
\label{sizedata}
\end{align}
for any $s\in \R$.
We define $\ve{u}_{0,n}:=\ve{u}_0+\ve{\phi}_n$.
For each $n\in \mathbb{N}$, Lemma~\ref{LEM:local} implies that there exists a unique solution $u_{n}\in C([0,T];\FL^{1}(\M))$ to \eqref{duhamel} with $(u_{n},\dt u_{n})\vert_{t=0}=\ve{u}_{0,n}$ and admitting the power series expansion
\begin{align}
u_{n}=\sum_{j=0}^{\infty} \Xi_{j}(\vec{u}_{0,n})=\sum_{j=0}^{\infty} \Xi_{j}(\vec{u}_{0}+\vec{\phi}_{n})
\label{powerseries}
\end{align}
on $[0,T]$, provided 
\begin{align}
0<T\ll \big( \|\vec{u}_{0}\|_{\FLv^{1}}+RA^{d}\big)^{-\frac{k-1}{2}}. \label{localtime}
\end{align}

We now state some key estimates for the multilinear expressions $\Xi_{j}(\vec{u}_{0,n})$. The proofs are essentially the same as in \cite[Lemma 3.2 and Lemma 3.3]{O17} and are included for completeness.

\begin{lemma}\label{LEM:multlin2}
For any $s<0$ and $j\in \mathbb{N}$, the following estimates hold: 
\begin{align}
\|\vec{u}_{0,n}-\vec{u}_{0}\|_{\H^{s}} &\lesssim RA^{\frac{d}{2}}N^{s}, \label{hs1} \\
\|\Xi_{0}(\vec{u}_{0,n})(t)\|_{H^{s}} &\lesssim 1+RA^{\frac{d}{2}}N^{s},  \label{hs2} \\
\|\Xi_{1}(\vec{u}_{0,n})(t)-\Xi_{1}(\vec{\phi}_{n})(t)\|_{H^s} &\lesssim t^{2}\|\vec{u}_{0}\|_{ \H^{0}}R^{k-1}A^{d(k-1)},  \label{hs3}  \\
\|\Xi_{j}(\vec{u}_{0,n})(t)\|_{H^s} &\lesssim C^j t^{2j}(RA^{d})^{(k-1)j} ( \|\vec{u}_0\|_{\H^0}+Rf_{s}(A)), \label{hs4}
\end{align} 
where 
\begin{equation}
f_{s}(A):=
\begin{cases}
 1 & \textup{if} \,\, \,s<-\frac{d}{2}, \\
 \left( \log A \right)^{\frac{1}{2}}& \textup{if} \,\,\, s=-\frac{d}{2}, \\
 A^{\frac{d}{2}+s} & \textup{if} \,\, \, s>-\frac{d}{2}
\end{cases}
\notag
\end{equation} 
and $0<t\leq 1$.
\end{lemma}

\begin{proof}
The proofs of \eqref{hs1} and \eqref{hs2} follow immediately from $\vec{\phi}_{n}=\vec{u}_{0,n}-\vec{u}_0$,  \eqref{sizedata}, and \eqref{linestHs}. 
By the multilinearity of $\I$, we have 
\begin{align}
\Xi_{1}(\vec{u}_{0,n})(t)-\Xi_{1}(\vec{\phi}_{n})(t)= \sum_{\vec{\psi}_1, \ldots, \vec{\psi}_{k}} \I[ S(t)\vec{\psi}_1,\ldots, S(t)\vec{\psi}_{k}], \label{diff1}
\end{align}
where the sum is over all choices of $\vec{\psi}_{j}\in \{ \vec{u}_0, \vec{\phi}_{n}\}$ with at least one appearance of $\vec{u}_{0}$.
Since $s<0$, \eqref{diff1}, Young's inequality, \eqref{linestHs}, and \eqref{linest} imply
\begin{align*}
\|\Xi_{1}(\vec{u}_{0,n})(t)-\Xi_{1}(\vec{\phi}_{n})(t)\|_{H^s} & \leq \|\Xi_{1}(\vec{u}_{0,n})(t)-\Xi_{1}(\vec{\phi}_{n})(t)\|_{L^2} \\
& \leq \sum_{\vec{\psi}_1, \ldots, \vec{\psi}_{k}} \| \I[ S(t)\vec{\psi}_1,\ldots, S(t)\vec{\psi}_{k}] \|_{L^2} \\
& \les t^{2}\|\vec{u}_{0}\|_{\H^{0}} (\|\vec{u}_0\|_{\FLv^1}^{k-1}+\|\vec{\phi}_{n}\|_{\FLv^1}^{k-1}).
\end{align*}
Using \eqref{sizedatau} and \eqref{sizedata}, we obtain \eqref{hs3}.

We now prove \eqref{hs4}. By the triangle inequality, we have 
\begin{align}
\|\Xi_{j}(\vec{u}_{0,n})(t)\|_{H^s} \leq \|\Xi_{j}(\vec{u}_{0,n})(t)-\Xi_{j}(\vec{\phi}_{n})(t)\|_{H^s}+\|\Xi_{j}(\vec{\phi}_{n})(t)\|_{H^{s}},
\label{xiju}
\end{align}
and thus we reduce to proving estimates for the two terms on the right hand side of the above.

From \eqref{data} and \eqref{omega}, $\supp \mathcal{F} [ S(t)\vec{\phi}_{n}]$ is contained within at most four disjoint cubes of volume approximately $A^{d}$. 
Thus, for each fixed $\mathcal{T}\in {\bf{T}}(j)$, the support of $\mathcal{F} [ \Psi_{\vec{\phi}_{n}}(\mathcal{T})]$ is contained in at most $4^{(k-1)j+1}$ cubes of volume approximately $A^d$.
Hence, for some $c,C>0$, we have 
\begin{align*}
| \supp\mathcal{F} [ \Xi_{j}(\ve{\phi}_{n})(t)] |\leq C^{j}A^{d}\leq |C^{\frac{j}{d}}Q_{A}|.
\end{align*}
As $s<0$, $\jb{\xi}^{s}$ is decreasing in $|\xi|$ and using \eqref{FLinfty}, \eqref{linestHs}, \eqref{linest}, and \eqref{sizedata}, we have 
\begin{align}
\begin{split}
\| \Xi_{j}(\vec{\phi}_n)(t)\|_{H^{s}} & \leq \|\jb{\xi}^{s}\|_{L^2(\supp\mathcal{F} [ \Xi_{j}(\ve{\phi}_{n})(t)])} \|  \Xi_{j}(\vec{\phi}_n)(t)\|_{\FL^{\infty}} \\
& \les C^{j}\|\jb{\xi}^{s}\|_{L^2(C^{\frac{j}{d}}Q_A)} t^{2j}(RA^{d})^{(k-1)j-1}(RA^{\frac{d}{2}})^{2} \\
& \les C^{j}t^{2j}(RA^{d})^{(k-1)j}Rf_{s}(A).
\end{split}
\label{xijphi}
\end{align}
Meanwhile, by considerations similar to \eqref{diff1}, we have 
\begin{align}
\begin{split}
\|\Xi_{j}(\vec{u}_{0,n})(t)-\Xi_{j}(\vec{\phi}_{n})(t)\|_{H^s} & \leq \|\Xi_{j}(\vec{u}_{0,n})(t)-\Xi_{j}(\vec{\phi}_{n})(t)\|_{L^2} \\
& \leq C^{j}t^{2j}\|\vec{u}_{0}\|_{\H^0}(\|\vec{u}_0 \|_{\FLv^{1}}^{(k-1)j}+\|\vec{\phi}_{n}\|_{\FLv^{1}}^{(k-1)j})\\
&\leq C^{j}t^{2j}\|\vec{u}_0\|_{\H^0}(RA^{d})^{(k-1)j}.
\end{split}
\label{xijphi2}
\end{align}
Thus, \eqref{hs4} follows from \eqref{xiju}, \eqref{xijphi}, and \eqref{xijphi2}.
\end{proof}

We now recall the following bounds on convolutions of characteristic functions of cubes:
\begin{lemma} For any $a,b,\xi\in \ft{\M}$ and $A\geq 1$, we have
\begin{align}
c_{d}A^{d}\chi_{a+b+Q_A}(\xi) \leq\chi_{a+Q_A}\ast \chi_{b+Q_A}(\xi)\leq  C_{d}A^{d}\chi_{a+b+Q_{2A}}(\xi).
\label{conv}
\end{align}
\end{lemma}

In the following proposition, we identify that the first multilinear term in the Picard expansion is culpable for the instability in Proposition~\ref{PROP:NI}.

\begin{proposition}\label{PROP:INSTAB}
Let $k\geq 2$ be an integer and $s<0$.
Let $\vec{\phi}_{n}$ be as in \eqref{data}.
Then, for $(AN)^{-\frac 12} \ll T\ll A^{-1}$, we have 
\begin{align}
\| \Xi_{1}(\ve{\phi}_n)(T)\|_{H^{s}}
\ges T^{2}R^k A^{d(k-1)} A^{\frac{d}{2}+s}.
\label{lowerbound}
\end{align}
\end{proposition}

\begin{proof}

\noi
To simplify notation, we write 
\begin{align}
\Gamma:= \bigg\{ (\xi_1,\ldots, \xi_k) \in \ft{\M}^k \, : \, \sum_{j=1}^{k}\xi_{j}=\xi \bigg\} \qquad \text{and} \qquad d\xi_{\G}:=d\xi_{1}\cdots d\xi_{k-1}. 
\notag
\end{align}
Restricting $|\xi|\les A$ and using $A\ll N$, \eqref{data}, and product-to-sum formulas, we have
\begin{align*}
\mathcal{F}[ \Xi_{1}(\ve{\phi}_n)(T) ](\xi)
&= \int_{0}^{T} \frac{\sin((T-t)|\xi|)}{|\xi|} \mathcal{F} \big[ (S(t)(\phi_{0,n},0))^{k}\big] dt\\
&= \int_{0}^{T} \frac{\sin((T-t)|\xi|)}{|\xi|}\int_{\G} \prod_{m=1}^{k} \cos(t|\xi_m|)\ft \phi_{0,n}(\xi_m) d\xi_{\G}dt \\
& = R^k 
\int_{0}^{T} \frac{ \sin((T-t)|\xi|)}{|\xi|} \int_{\G} \prod_{m=1}^{k} \cos(t|\xi_{m}|)\chi_{\O}(\xi_{m}) dt d\xi_{\G} \\
& = \frac{R^k}{2^k}
\sum_{\substack{(\eta_1, \dots, \eta_k) \in \Si^k \\ \eta_1+ \dots + \eta_k =0}}
\sum_{(\eps_1, \dots, \eps_k) \in \{ -1,1\}^k}
\int_{0}^{T} \frac{ \sin((T-t)|\xi|)}{|\xi|} \\
&\quad \times
\int_{\G}\cos \bigg( t \sum_{j=1}^k \eps_j |\xi_j| \bigg) \prod_{j=1}^{k} \chi_{\eta_j + Q_A}(\xi_j) d\xi_{\G} dt.
\end{align*}
For each fixed $\eta:=(\eta_1,\ldots,\eta_k)\in \Sigma^{k}$ with $\eta_1+\dots+\eta_k=0$, we split the inner summation into two parts:
\begin{align*}
\sum_{(\eps_1, \dots, \eps_k) \in \{ -1,1\}^k} = \sum_{(\eps_1, \dots, \eps_k) \in S_1 (\eta)} +\sum_{(\eps_1, \dots, \eps_k) \in S_2(\eta) } 
\end{align*}
where 
\begin{align*}
S_{1}(\eta):= \bigg\{ (\eps_1, \dots, \eps_k) \in \{ -1,1\}^k \, : \, \sum_{j=1}^{k} \eps_{j}|\eta_{j}|=0 \bigg\},
\end{align*}
 $S_{2}(\eta) := \{-1,1\}^{k} \setminus S_1(\eta)$ and we write 
 \begin{align}
\mathcal{F}& [\Xi_{1}(\ve{\phi}_n)(T) ](\xi)
=\frac{R^{k}}{2^k}\sum_{\substack{\eta=(\eta_1, \dots, \eta_k) \in \Si^k \\ \eta_1+ \dots + \eta_k =0}} \Big( I_{1}(\eta,\xi,T)+I_{2}(\eta,\xi,T) \Big). 
\label{decomp}
\end{align}
Note that the set $S_{1}(\eta)$ is non-empty.
For fixed $\xi_{j}\in \eta_{j}+Q_{A}$ and then $(\eps_1, \dots, \eps_k) \in S_1 (\eta)$,
\begin{align}
\bigg\vert \sum_{j=1}^{k} \eps_{j}|\xi_{j}| \bigg\vert
=
\bigg\vert \sum_{j=1}^{k} \eps_{j}\big(|\xi_{j}|-|\eta_j|\big) \bigg\vert
\le
\sum_{j=1}^{k} |\xi_{j} - \eta_j| 
 \les A.
\label{S1prop}
\end{align}
Then, it follows from \eqref{S1prop} that 
\begin{align}
\cos \bigg( t \sum_{j=1}^k \eps_j |\xi_j| \bigg)
\geq \frac{1}{2} \label{cosbd2}
\end{align}
for $0<t<T\ll A^{-1}$.
Moreover, we have
\begin{align}\label{sinbd}
\frac{\sin ((T-t)|\xi|)}{|\xi|} \ges T-t
\end{align}
for $0<t<T\ll A^{-1}$ and $|\xi| \les A$.
Using \eqref{decomp}, \eqref{conv}, \eqref{cosbd2}, and \eqref{sinbd},
we obtain
\begin{align*}
I_{1}(\eta,\xi,T) &\ges \sum_{(\eps_1, \dots, \eps_k) \in S_1 (\eta)}\int_{0}^{T} (T-t)dt 
\int_{\G} \prod_{j=1}^{k} \chi_{\eta_j +Q_{A}}(\xi_j)d\xi_{\G} \\
& \ges T^{2}A^{d(k-1)}\chi_{Q_{A}}(\xi)
\end{align*}
and hence
\begin{align}
\frac{R^{k}}{2^{k}}\sum_{\substack{\eta=(\eta_1, \dots, \eta_k) \in \Si^k \\ \eta_1+ \dots + \eta_k =0}}
I_{1}(\eta,\xi,T) 
\ges T^{2}R^{k}A^{d(k-1)}\chi_{Q_{A}}(\xi)
 \label{I1con}
\end{align}
for $1\leq A\ll N$ and $0<T\ll A^{-1}$.

We now turn to the contribution from $I_{2}(\eta,\xi,T)$. We observe that 
for each fixed $\eta=(\eta_1,\ldots,\eta_k)$, $(\eps_1, \dots, \eps_k) \in S_{2}(\eta)$, and $\xi_j \in \eta_{j}+Q_{A}$, 
\begin{align}
\bigg\vert \sum_{j=1}^{k} \eps_{j}|\xi_j| \bigg\vert \sim N. \label{S2prop}
\end{align}
In view of $A\ll N$, the upper bound is obvious. For the lower bound,
the reverse triangle inequality yields $||\xi_j|-|\eta_j|| \le |\xi_j-\eta_j| \les A$ for $\xi_j \in \eta_{j}+Q_{A}$ and hence, we have
\begin{align*}
\sum_{j=1}^{k} \eps_{j}|\xi_j|
= \sum_{j=1}^{k} \eps_j |\eta_j|  
+ \sum_{j=1}^{k} \eps_{j} ( |\xi_j| - |\eta_j|)
=\sum_{j=1}^{k} \eps_j |\eta_j| 
+\mathcal{O}(A).
\end{align*}
It follows from \eqref{Sig} that
\begin{align*}
\bigg\vert \sum_{j=1}^{k} \eps_j |\eta_j| \bigg\vert \geq N
\end{align*}
for $(\eps_1, \dots, \eps_k) \in S_2(\eta)$,
which verifies \eqref{S2prop}. 
We therefore have 
\begin{align*}
I_{2}(\eta,\xi,T) =
\sum_{(\eps_1, \dots, \eps_k) \in S_2(\eta)}
&\frac{1}{2|\xi|}\int_{\G}\int_{0}^{T} 
\bigg[ \sin \bigg(  T|\xi|-t\bigg( |\xi|-\sum_{j=1}^{k} \eps_j |\xi_j|\bigg) \bigg)  \\
&+\sin\bigg( T|\xi|-t\bigg( |\xi|+\sum_{j=1}^{k} \eps_{j}|\xi_j| \bigg)\bigg) \bigg] dt
 \prod_{j=1}^{k}\chi_{\eta_j+Q_A}(\xi_j)d\xi_{\G}.
\end{align*}
If we further restrict $\xi\in \frac{A}{4}e_1 +Q_{\frac{A}{4}}$ and use \eqref{S2prop},
 \eqref{conv}, and $A \ll N$, we obtain
\begin{align*}
|I_{2}(\eta,\xi,T)| \les 2^k A^{-1}N^{-1}A^{d(k-1)}\chi_{Q_{kA}}(\xi),
\end{align*}
which implies 
\begin{align}
\bigg\vert \frac{R^{k}}{2^{k}}\sum_{\substack{\eta=(\eta_1, \dots, \eta_k) \in \Si^k \\ \eta_1+ \dots + \eta_k =0}}
 I_{2}(\eta,\xi,T)\bigg\vert \les A^{-1}N^{-1}R^{k}A^{d(k-1)}\chi_{Q_{kA}}(\xi). \label{I2con}
\end{align}

Returning to \eqref{decomp} and using \eqref{I1con}, \eqref{I2con} and imposing $T^{2}AN\gg \max(1,C)$, we obtain
\begin{align}
\| \Xi_{1}(\ve{\phi}_n)(T)\|_{H^{s}}
&\ge \| \jb{\xi}^s \Ft [\Xi_{1}(\ve{\phi}_n)](T)\|_{L^2_\xi (\frac{A}{4}e_1 +Q_{\frac{A}{4}})} \notag\\
&\ges \Big( T^2 - CA^{-1} N^{-1} \Big) R^k A^{d(k-1)} \| \jb{\xi}^{s}\|_{L^{2}_{\xi}( \frac{A}{4}e_1 +Q_{\frac{A}{4}})} \notag\\
&\ges T^2 R^k A^{d(k-1)} A^{\frac{d}{2}+s}, \notag
\end{align}
which shows \eqref{lowerbound}.
\end{proof}

\begin{remark}\rm \label{REM:NLKGlb}
The same result as in Proposition \ref{PROP:INSTAB} is valid for \eqref{NLKG}.
Indeed, since the linear solution of \eqref{NLKG} is written as
\[
S(t)(u_0,u_1)= \cos(t\jb{\nb}) u_0 + \frac{\sin (t \jb{\nb})}{\jb{\nb}} u_1,
\]
it suffices to replace $|\xi|$ in the proof with $\jb{\xi}$.
More precisely, it follows from $\jb{\xi_j}-|\xi_j| = \frac1{\jb{\xi_j}+|\xi_j|}$ and \eqref{S1prop} that
\begin{align}
\bigg\vert \sum_{j=1}^{k} \eps_{j} \jb{\xi_{j}} \bigg\vert
&\le \bigg\vert \sum_{j=1}^{k} \eps_{j} |\xi_{j}| \bigg\vert + \sum_{j=1}^{k} (\jb{\xi_j}-|\xi_j|) \notag\\
&\les A,
\notag
\end{align}
for $\eta=(\eta_1,\ldots,\eta_k) \in \Si^k$, $(\eps_1, \dots, \eps_k) \in S_1 (\eta)$, and $\xi_{j}\in \eta_{j}+Q_{A}$.
Similarly, from \eqref{S2prop}, we have
\begin{align}
\bigg\vert \sum_{j=1}^{k} \eps_{j}|\xi_j| \bigg\vert
&\ge \bigg\vert \sum_{j=1}^{k} \eps_{j} |\xi_{j}| \bigg\vert - \sum_{j=1}^{k} (\jb{\xi_j}-|\xi_j|) \notag\\
&\sim N,
\notag
\end{align}
for $\eta=(\eta_1,\ldots,\eta_k) \in \Si^k$, $(\eps_1, \dots, \eps_k) \in S_{2}(\eta)$, and $\xi_j \in \eta_{j}+Q_{A}$.
Hence, from the same argument as in the proof of Proposition \ref{PROP:INSTAB},
the first multilinear term in the Picard expansion for \eqref{NLKG} satisfies \eqref{lowerbound}.
\end{remark}

\subsection{Proof of Proposition \ref{PROP:NI}}
\label{SUBSEC:32}

In order to prove Proposition~\ref{PROP:NI}, it suffices to show, given $n\in \mathbb{N}$, the following hold:
\begin{align*}
 \textup{(i)} & \quad RA^{\frac d2}N^s \ll \frac 1n, \\ 
 \textup{(ii)} & \quad T^2 R^{k-1}A^{d(k-1)} \ll 1,  \\ 
 \textup{(iii)} & \quad T^{2}R^{k}A^{d(k-1)}A^{\frac{d}{2}+s} \gg n,\\ 
 \textup{(iv)} & \quad T^2 R^k A^{d(k-1)}A^{\frac{d}{2}+s} \gg T^{4}R^{2(k-1)+1}A^{2d(k-1)}f_{s}(A), \\
 \textup{(v)} & \quad (AN)^{-\frac{1}{2}}\ll T\ll \min \Big( A^{-1}, \frac 1n \Big),\\ 
 \textup{(vi)} & \quad \| \ve{u}_0 \|_{\H^0}\ll RA^{\frac{d}{2}+s}, \,\, A\ll N, \,\, \| \ve{u}_0\|_{\FLv^{1}}\ll RA^d
 \end{align*}
for some particular choices of $R,T$, and $N$ all depending on $n$. While we could also allow $A$ to depend on $N$ as well, it turns out that we can choose $A$ independently of $N$ (and hence $n$) as we see below.
Notice that condition (iv) reduces to satisfying 
\begin{align*}
\textup{(iv')} & \quad T^{2}R^{k-1}A^{d(k-1)}f_{s}(A) \ll A^{\frac{d}{2}+s}.
\end{align*}
 The conditions (ii) and the last of (vi) ensure that the power series expansion \eqref{powerseries} is valid on $[0,T]$, where $T$ must satisfy \eqref{localtime}.

We now indicate how establishing (i) through (vi) suffices to prove Proposition~\ref{PROP:NI}. 
When (ii) and (vi) hold, it follows from \eqref{data}, \eqref{omega}, \eqref{sizedatau}, \eqref{sizedata}, and Lemma~\ref{LEM:local}, that for each $n\in \mathbb{N}$, there is a unique solution $u_{n} \in C([0,T];\FL^{1}(\M))$ to \eqref{duhamel} with $(u_n , \dt u_n)|_{t=0}=\vec{u}_{0,n}$ and such that the power series expansion \eqref{powerseries} converges on $[0,T]$.
By \eqref{hs1}, condition (i) ensures the first expression in \eqref{nexplosion}. From \eqref{hs4}, (ii), and (vi), we have
\begin{align}
\begin{split}
\Bigg\| \sum_{j=2}^{\infty} \Xi_{j}(\vec{u}_{0,n})(T) \Bigg\|_{H^{s}} & \les Rf_{s}(A)\sum_{j=2}^{\infty} (CT^{2}R^{k-1}A^{d(k-1)})^{j} \\
& \les T^{4}R^{2(k-1)+1}A^{2d(k-1)}f_{s}(A).
\end{split}
\label{tail}
\end{align}
Then, from Proposition~\ref{PROP:INSTAB} (thus requiring (v)), \eqref{hs2}, \eqref{hs3}, \eqref{tail}, (iii), (iv), and (vi), we have 
\begin{align*}
\|u_{n}(T)\|_{H^s}& \geq \|\Xi_{1}(\vec{\phi}_{n})(T)\|_{H^s}-\|\Xi_{0}(\vec{u}_{0,n})\|_{H^s} \\
& \hphantom{XX}-\|\Xi_{1}(\vec{u}_{0,n})(T)-\Xi_{1}(\vec{\phi}_{n})(T)\|_{H^s}-\bigg\| \sum_{j=2}^{\infty}\Xi_{j}(\vec{u}_{0,n})(T)\bigg\|_{H^s} \\
&\gtrsim T^{2}R^{k}A^{d(k-1)}A^{\frac{d}{2}+s}- (1+RA^{\frac{d}{2}}N^{s}) -T^2 \|\vec{u}_{0}\|_{\H^{0}}R^{k-1}A^{d(k-1)} \\
& \hphantom{X} -T^{4}R^{2(k-1)+1}A^{2d(k-1)}f_{s}(A) \\
& \sim T^{2}R^{k}A^{d(k-1)}A^{\frac{d}{2}+s} \gg n,
\end{align*}
which establishes the second expression in \eqref{nexplosion} and hence Proposition~\ref{PROP:NI}.
It remains to show, given $n\in \mathbb{N}$, we can choose $A,R$, and $T$ depending on $N$ and then $N=N(n)\gg 1$ so that conditions (i) through (vi) hold. 

\medskip

\noi 
$\bullet$ \textbf{Case 1:} $-\frac1{k-1} \le s<0$.

\medskip

\noi
We choose 
\begin{align}
A=10, \qquad R=N^{-s-\dl}, \qquad T=N^{\frac {k-1}2( s + \frac \dl 2)},
\notag
\end{align}
for $0<\dl\ll1 $ sufficiently small so that
\begin{align*}
-s>\frac{k+1}{2}\dl.
\end{align*}
Our choice of $A$ ensures that $f_{s}(A)\sim A^{\frac{d}{2}+s}$ so (iv') essentially reduces to verifying (ii).
Then, we have 
\begin{align*}
RA^{\frac d2}N^{s}& \sim N^{-\dl}\ll \frac{1}{n}, \\ 
T^{2}R^{k-1}A^{d(k-1)}& \sim N^{-\frac{k-1}{2} \dl}\ll 1, \\
T^{2}R^{k}A^{d(k-1)} A^{\frac{d}{2}+s} & \sim N^{-s-\frac{k+1}{2}\dl}\gg n, \notag \\
TA &\sim N^{\frac{k-1}{2}(s+\frac{\dl}{2})}\ll \frac 1n, \\
T^{2}AN& \sim N^{(k-1)(s+\frac{1}{k-1}+\frac{\dl}{2})}\gg 1, 
\end{align*}
since $k \ge 2$ and $-\frac{1}{k-1}\leq s<0$.

\medskip

\noi 
$\bullet$ \textbf{Case 2:} $s <-\frac1{k-1}$.

\medskip

\noi
This case follows from Case 1 with $s=-\frac1{k-1}$.
More precisely, we choose
\begin{align}
A=10, \qquad R=N^{\frac 1{k-1}-\dl}, \qquad T=N^{-\frac 12 +\frac {k-1}4\dl}
\notag
\end{align}
for $0<\dl < \frac 2{k^2-1}$.
Then, we obtain
\begin{align*}
RA^{\frac d2}N^{s}& \sim N^{s+\frac1{k-1}-\dl}\ll \frac{1}{n}, \\ 
T^{2}R^{k-1}A^{d(k-1)}& \sim N^{-\frac{k-1}{2} \dl}\ll 1, \\
T^{2}R^{k}A^{d(k-1)} A^{\frac{d}{2}+s} & \sim N^{\frac1{k-1}-\frac{k+1}{2}\dl}\gg n, \notag \\
TA &\sim N^{-\frac 12 + \frac{k-1}{4} \dl}\ll \frac 1n, \\
T^{2}AN& \sim N^{\frac{k-1}2\dl}\gg 1, 
\end{align*}
since $k \ge 2$ and $s < - \frac1{k-1}$.

\begin{remark}\rm \label{RMK:dt}
In this remark, we analyse the growth of the Sobolev norm of $\dt u$.
First, note that the following analogue of Lemma \ref{LEM:Xiestimates} holds:
Given $\ve \phi \in \FLv^{1}(\M)$, $j\in \mathbb{N}$, $\vec{\psi}\in \FLv^{1}(\M) \cap \H^{0}(\M)$, and $0<T\leq 1$, we have 
\begin{align}
\| \dt \Xi_{j}(\ve{\phi})(T)\|_{\FL^{-1,1}} &\leq C^{j}T^{2j-1}\|\ve{\phi}\|_{\FLv^{1}}^{(k-1)j+1}, \label{FL1t}\\
\| \dt \Xi_{j}(\ve{\psi})(T)\|_{\FL^{-1,\infty}} &\leq C^{j}T^{2j-1}\|\ve{\psi}\|_{\FLv^{1}}^{(k-1)j-1}  \|\ve{\psi}\|_{\H^{0}}^{2}.
\notag
\end{align}
\noi
We use the same notation as in Subsection \ref{SUBSEC:31}.
Namely, we fix $\vec{u}_0\in \mathcal{S}(\M)\times \mathcal{S}(\M)$.
We choose $\ve{\phi}_{n}=(\phi_{0,n},0)$, where $\phi_{0,n}$ given by \eqref{data}.
Moreover, for $n\in \mathbb{N}$, we will choose $N(n)\gg 1$ and $R=R(N)\geq 1$ and $A\gg 1$, so that 
 \begin{align}
\|\ve{u}_0\|_{\H^{0}}\ll RA^{\frac{d}{2}+s}, \qquad  A\ll N, \qquad \text{and}  \qquad\|\ve{u}_0\|_{\FLv^{1}}\ll RA^{d},
\label{sizedatau2}
\end{align}
which is (vi) in Subsection \ref{SUBSEC:32}.
With $\ve{u}_{0,n}:=\ve{u}_0+\ve{\phi}_{n}$, for each $n\in \mathbb{N}$.
By \eqref{FL1t},
the solution $u_{n}\in C([0,T];\FL^{1}(\M))$ to \eqref{duhamel} with $(u_{n},\dt u_{n})\vert_{t=0}=\ve{u}_{0,n}$ is also in $C^1([0,T];\FL^{-1,1}(\M))$ and admits the power series expansion
\begin{align*}
\dt u_{n}=\sum_{j=0}^{\infty} \dt \Xi_{j}(\vec{u}_{0,n})=\sum_{j=0}^{\infty} \dt \Xi_{j}(\vec{u}_{0}+\vec{\phi}_{n})
\end{align*}
provided \eqref{localtime} holds.
Analogous to Lemma~\ref{LEM:multlin2}, we have the following multilinear estimates:
\begin{align}
\| \dt \Xi_{1}(\ve{u}_{0,n})(t)- \dt \Xi_{1}(\ve{\phi}_{n})(t)\|_{H^{s-1}} &\lesssim t\|\vec{u}_{0}\|_{ \H^{0}}R^{k-1}A^{d(k-1)},  \label{hs132}  \\
\|\dt \Xi_{j}(\ve{u}_{0,n})(t)\|_{H^{s-1}} &\lesssim C^j t^{2j-1}(RA^{d})^{(k-1)j} ( \|\ve{u}_0\|_{\H^0}+Rf_{s}(A)) \label{hs442}
\end{align} 
for $j \in \NB$ and $0<t \le 1$.
The last estimate \eqref{hs442}, under \eqref{sizedatau2}, implies
\begin{align}
\begin{split}
\bigg\| \sum_{j=2} \dt \Xi_{j}(\ve{u}_{0,n})(T) \bigg\|_{H^{s-1}} 
&\les T^{3}R^{2(k-1)+1}A^{2d(k-1)}f_{s}(A),
\label{remaining2}
\end{split}
\end{align}
 provided $T^{2}(RA^d)^{k-1}\ll 1$.
Furthermore, following the proof of Proposition~\ref{PROP:INSTAB}, for integer $k\geq 2$, $s<0$ and $N^{-1}\ll T\ll A^{-1}$, we have 
\begin{align}
\|\dt \Xi_{1}(\ve{\phi}_{n})(T)\|_{H^{s-1}}\ges TR^{k}A^{d(k-1)}A^{-1}A^{\frac{d}{2}+s}. \label{lowerbound2}
\end{align}
Now, we choose the parameters $A, T, R$ and $N=N(n)$ as in Section~\ref{SUBSEC:32}.
In particular, \eqref{lowerbound2} with $A=10$ yields that
\begin{align}
\|\dt \Xi_{1}(\ve{\phi}_{n})(T)\|_{H^{s-1}}
\ges TR^{k}.
\label{lowerbound3}
\end{align}
Hence, from \eqref{hs132}, \eqref{remaining2}, and \eqref{lowerbound3},
we obtain
\begin{align*}
\|\dt u_{n}(t_{n})\|_{H^{s-1}}\gg n.
\end{align*}

\end{remark}

\section{Norm inflation in Fourier-amalgam spaces}\label{app:amalgam}
In this section, we detail the necessary modifications to the proof of Theorem~\ref{THM:NI} needed to prove Theorem~\ref{THM:FLq}. We apply the same overall approach as detailed in Sections~\ref{SEC:2} and \ref{SEC:3}. As we saw earlier, the proof of Theorem~\ref{THM:FLq} reduces to proving norm inflation but for regular initial data:
\begin{proposition} \label{PROP:NI2}
Let $d\in \NB$, $k\geq 2$ be an integer, $1\leq p,q\leq \infty$ and $s<0$.
Fix $u_0,u_1\in \mathcal{S}(\M)$. Then, for any $n\in \mathbb{N}$, there exists a smooth solution $u_n$ to the NLW~\eqref{NLW1} and $t_{n}\in (0, \frac{1}{n})$ such that
\begin{align*}
\|(u_n (0), \dt u_{n}(0))-(u_0,u_1)\|_{X^{p,q}_{s}(\M)}<\frac 1n
\qquad \text{and} \qquad
\|u_{n}(t_n)\|_{X^{p,q}_{s}(\M)}>n,
\end{align*}
where the space $X^{p,q}_{s}(\M)$ is defined in \eqref{Xspace}.
\end{proposition}

In order to prove Proposition~\ref{PROP:NI2}, we only need to make modifications to statements and estimates appearing in Section~\ref{SEC:3}. Namely, we begin with the exact same setup provided by Section~\ref{SEC:2}, in particular, the power series expansion \eqref{uexpand} assured by Lemma~\ref{LEM:local}.

Given $A,R,N$ to be chosen later, depending on $n\in \NB$ (if necessary), we set $\ve{\phi}_{n}=(\phi_{0,n}, 0)$ exactly as in \eqref{data}, \eqref{omega} and \eqref{Sig}. 
With $\mathcal{X}^{p,q}_{s}(\M):=X^{p,q}_{s}(\M)\times X^{p,q}_{s-1}(\M)$, we have 
\begin{align}
\| \ve{\phi}_{n}\|_{\mathcal{X}^{p,q}_{s}(\M)}\sim RA^{\frac{d}{q}}N^{s},
\label{datasize1}
\end{align}
with the understanding that the right hand side above is $RN^s$ if $q=\infty$.
While \eqref{datasize1} is obvious when $p=q$ (since $\mathcal{X}^{q,q}_{s}(\M)=\FLv^{s,q}(\M)$), we
provide a few details for the case when $p\neq q$ and $\M=\R^{d}$ to give a flavour of the arguments to follow relating to \eqref{fw}. We have $\|\ve{\phi}_{n}\|_{\mathcal{X}^{p,q}_{s}(\R^d)}=\| \phi_{0,n}\|_{\fw^{p,q}_{s}(\R^d)}$.
For fixed $n\in \Z^d$, 
\begin{align*}
\| \ft \phi_{0,n}(\xi)\chi_{n+Q}(\xi)\|_{L^p_{\xi}}&=R\, \bigg\| \sum_{\eta \in \Sigma} \chi_{(\eta+Q_A)\cap (n+Q)}(\xi) \bigg\|_{L^{p}_{\xi}} \\
 &=R \bigg( \sum_{\eta\in \Sigma} \text{meas}( (\eta+Q_A)\cap (n+Q)) \bigg)^{\frac{1}{p}}.
\end{align*}
For fixed $n\in \Z^{d}$, the above sum is non-zero if and only if there exists $\mu\in \{\pm1,\pm 2\}$ such that 
$|n+\mu Ne_1| \sim A$. Since $A\ll N$, the sets $\{ |n+\mu Ne_1|\sim A\}_{\mu}$ are disjoint and hence,
\begin{align*}
\| \phi_{0,n}\|_{\fw^{p,q}_{s}(\R^d)} \les R  \bigg(  \sum_{\mu\in \{\pm 1,\pm2\}} \sum_{|n+\mu Ne_1|\sim A} \jb{n}^{sq}\bigg)^{\frac{1}{q}} \sim RA^{\frac{d}{q}}N^{s}.
\end{align*}
The lower bound follows by bounding below the sum in $\eta$ above by just one particular choice of $\eta$. Similarly, we have
\begin{align*}
\|\phi_{0,n}\|_{\fw^{p,\infty}_{s}(\R^d)} \sim R\max_{\mu\in \{\pm 1,\pm2\}} \max_{|n+\mu Ne_1|\sim A} \jb{n}^{s} \sim RN^{s}.
\end{align*} 

We have the following analogue of Lemma~\ref{LEM:multlin2}.
\begin{lemma}\label{LEM:multlin3}
For any $s<0$ and $j\in \mathbb{N}$, the following estimates hold: 
\begin{align}
\|\vec{u}_{0,n}-\vec{u}_{0}\|_{\mathcal{X}^{p,q}_{s}(\M)} &\lesssim RA^{\frac{d}{q}}N^{s}, \notag \\
\|\Xi_{0}(\vec{u}_{0,n})(t)\|_{X^{p,q}_{s}(\M)} &\lesssim 1+RA^{\frac{d}{q}}N^{s},  \notag\\
\|\Xi_{1}(\vec{u}_{0,n})(t)-\Xi_{1}(\vec{\phi}_{n})(t)\|_{X^{p,q}_{s}(\M)} &\lesssim t^{2}\|\vec{u}_{0}\|_{\mathcal{X}^{p,q}(\M)}R^{k-1}A^{d(k-1)},  \label{hs32}  \\
\|\Xi_{j}(\vec{u}_{0,n})(t)\|_{X^{p,q}_{s}(\M)} &\lesssim C^j t^{2j}(RA^{d})^{(k-1)j} ( \|\vec{u}_0\|_{\mathcal{X}^{p,q}(\M)}+Rf_{s,q}(A)), \label{hs42}
\end{align} 
where $0<t\leq 1$ and where we define
\begin{equation}
f_{s,q}(A):=
\begin{cases}
 1 & \textup{if} \,\, \,s<-\frac{d}{q}, \\
 \left( \log A \right)^{\frac{1}{q}}& \textup{if} \,\,\, s=-\frac{d}{q}, \\
 A^{\frac{d}{q}+s} & \textup{if} \,\, \, s>-\frac{d}{q},
\end{cases}
\notag
\end{equation}
for $1\leq q<\infty$  
and $f_{s,\infty}(A)= 1$.
\end{lemma}

The proof of Lemma~\ref{LEM:multlin3} follows exactly the same ideas as the proof of Lemma~\ref{LEM:multlin2}. The only noteworthy modification is in showing
\eqref{hs32} and \eqref{hs42}.
For this, we use the following product estimate:
\begin{align}
\| fg\|_{X_{0}^{p,q}(\M)} \les \|f\|_{X_{0}^{p,q}(\M)}\|g\|_{\FL^{1}(\M)}, \label{prodest}
\end{align}
for any $1\leq p, q\leq \infty$. 
When $\M=\T^{d}$, \eqref{prodest} follows by Young's inequality, while if $\M=\R^{d}$,
\eqref{prodest} follows from the uniform decomposition
\begin{align*}
\sum_{n\in \Z^{d}}\chi_{n+Q}(\xi)=1, \quad \text{for every} \quad \xi \in \R^d
\end{align*}
 and two applications of Young's inequality.
Similarly, we have the following analogue of Proposition~\ref{PROP:INSTAB}.
\begin{proposition}\label{PROP:INSTAB2}
Let $k\geq 2$ be an integer, $1\leq p,q \leq \infty$ and $s<0$.
Let $\vec{\phi}_{n}$ be as in \eqref{data}.
Then, for $(AN)^{-\frac 12} \ll T\ll A^{-1}$, we have 
\begin{align}
\| \Xi_{1}(\ve{\phi}_n)(T)\|_{X^{p,q}_{s}(\M)}
\ges T^{2}R^k A^{d(k-1)}A^{\frac{d}{q}+s}.
\notag
\end{align}
\end{proposition}

Now as in Subsection~\ref{SUBSEC:32}, Lemma~\ref{LEM:multlin3} and Proposition~\ref{PROP:INSTAB2} imply that Proposition~\ref{PROP:NI2} follows provided we can show that given $n\in \NB$, we can choose $A,R,T,$ and $N$ depending on $n$ such that the following conditions are satisfied:
\begin{align*}
 \textup{(i)} & \quad RA^{\frac dq}N^s \ll \frac 1n, \\ 
 \textup{(ii)} & \quad T^2 R^{k-1}A^{d(k-1)} \ll 1,  \\ 
 \textup{(iii)} & \quad T^{2}R^{k}A^{d(k-1)}A^{\frac{d}{q}+s} \gg n,\\ 
 \textup{(iv)} & \quad T^2 R^k A^{d(k-1)}A^{\frac{d}{q}+s}\gg T^{4}R^{2(k-1)+1}A^{2d(k-1)}f_{s, q}(A), \\
 \textup{(v)} & \quad (AN)^{-\frac{1}{2}}\ll T\ll \min \Big( A^{-1}, \frac 1n \Big),\\ 
 \textup{(vi)} & \quad \| \ve{u}_0 \|_{\mathcal{X}^{p,q}(\M)}\ll RA^{\frac{d}{q}+s}, \,\, A\ll N, \,\, \| \ve{u}_0\|_{\FLv^{1}}\ll RA^d.
 \end{align*}
We notice that the only places where the indices $p$ and $q$ appear above are with respect to the parameter $A$ (also in the implicit constants but they do not matter). Therefore, we can take $A=10$, and choose $R=R(N)$ and $T=T(N)$ exactly as in Cases 1 and 2 in Subsection~\ref{SUBSEC:32} above. This completes the proof of Proposition~\ref{PROP:NI2} and hence also the proof of Theorem~\ref{THM:FLq}.

\begin{ackno}\rm
The authors would like to thank Tadahiro Oh for suggesting this problem.
M.O. would like to thank the School of Mathematics at the University
of Edinburgh, where this manuscript was prepared, for its hospitality.
J.\,F.~was supported by The Maxwell Institute Graduate School in Analysis and its
Applications, a Centre for Doctoral Training funded by the UK Engineering and Physical
Sciences Research Council (grant EP/L016508/01), the Scottish Funding Council, Heriot-Watt
University and the University of Edinburgh.
Both authors also acknowledge support from Tadahiro Oh's ERC starting grant no. 637995 “ProbDynDispEq”.
M.O.~was supported by JSPS KAKENHI Grant numbers JP16K17624 and JP20K14342.
\end{ackno}

\end{document}